\documentclass[a4paper,10pt,reqno]{amsart}
\usepackage{latexsym}
\usepackage{tikz}
\usetikzlibrary{matrix,arrows,patterns}

\newcommand{\NN}{{\mathbb N}}
\newcommand{\ZZ}{{\mathbb Z}}
\newcommand{\CC}{{\mathbb C}}
\newcommand{\sym}{\mathfrak{S}}

\newcommand{\dual}[1]{{#1}^{\star}}

\newcommand{\opi}{\omega_{\pi}}
\newcommand{\osi}{\omega_{\sigma}}
\newcommand{\F}{\mathcal{F}}
\renewcommand{\O}{\Omega}
\newcommand{\pqbin}[2]{\genfrac{[}{]}{0pt}{0}{#1}{#2}}   
\newcommand{\tpqbin}[2]{\genfrac{[}{]}{0pt}{1}{#1}{#2}}  

\DeclareMathOperator{\SSF}{SSF}
\DeclareMathOperator{\maj}{\text{\sc maj}}
\DeclareMathOperator{\fix}{\mathrm{fix}}
\DeclareMathOperator{\sfix}{\mathrm{sfix}}
\DeclareMathOperator{\ssfix}{\mathrm{ssfix}}
\DeclareMathOperator{\des}{\mathrm{des}}
\DeclareMathOperator{\inv}{\text{\sc inv}}
\DeclareMathOperator{\ci}{\text{\sc ci}}
\DeclareMathOperator{\ca}{\text{\sc ca}}
\DeclareMathOperator{\mix}{\text{\sc mix}}
\DeclareMathOperator{\exc}{\mathrm{exc}}
\DeclareMathOperator{\lmax}{\mathrm{lmax}}
\DeclareMathOperator{\comp}{\mathrm{comp}}
\DeclareMathOperator{\stat}{\mathrm{stat}}
\DeclareMathOperator{\Int}{\mathrm{Int}}

\DeclareMathOperator{\al}{\mathrm{alph}}

\DeclareMathOperator{\SimSun}{\mathfrak{simsun}}
\DeclareMathOperator{\Andre}{\mathfrak{andr\mathaccent 19 e}}

\newcommand{\pattern}[4]{
  \raisebox{0.6ex}{
  \begin{tikzpicture}[scale=0.35, baseline=(current bounding box.center), #1]
    \foreach \x/\y in {#4}
      \fill[pattern=north east lines, pattern color=black!45] (\x,\y) rectangle +(1,1);
    \draw (0.01,0.01) grid (#2+0.99,#2+0.99);
    \foreach \x/\y in {#3}
      \filldraw (\x,\y) circle (5pt);
  \end{tikzpicture}}\;
}
\newcommand{\scl}{0.8}

\newcommand{\simsun}{
  \pattern{scale=\scl}{3}{1/3,2/2,3/1}{1/0,1/1,1/2,2/0,2/1,2/2}
}
\newcommand{\andre}{
  \pattern{scale=\scl}{3}{1/3,2/2,3/1}{1/0,1/3,2/0,2/3}
}

\theoremstyle{plain}
\newtheorem{theorem}{Theorem}
\newtheorem*{theorem*}{Theorem}
\newtheorem{corollary}[theorem]{Corollary}
\newtheorem*{corollary*}{Corollary}
\newtheorem{lemma}[theorem]{Lemma}
\newtheorem*{lemma*}{Lemma}
\newtheorem{proposition}[theorem]{Proposition}
\newtheorem*{proposition*}{Proposition}

\newtheorem*{conjecture*}{Conjecture}

\theoremstyle{definition}

\newtheorem*{definition*}{Definition}

\newtheorem*{example*}{Example}

\newtheorem*{problem*}{Problem}

\theoremstyle{remark}

\newtheorem*{remark*}{Remark}

\title[Mesh patterns and statistics as patterns]
      {Mesh patterns and the expansion of permutation 
        statistics as sums of permutation patterns}

\dedicatory{Dedicated to Doron Zeilberger on the occasion of his sixtieth birthday}

\author[P. Br\"and\'en]{Petter Br\"and\'en}
\author[A. Claesson]{Anders Claesson}
\thanks{PB is a Royal Swedish Academy of Sciences Research Fellow
  supported by a grant from the Knut and Alice Wallenberg
  Foundation. AC was supported by grant no. 090038011 from the
  Icelandic Research Fund. }
  \address{P. Br\"and\'en, Department of Mathematics, Stockholm University, 
    SE-106 91 Stockholm, Sweden}
\address{A. Claesson, Department of Computer and Information Sciences,
    University of Strathclyde, Glasgow, G1 1XH, UK}

\begin{document}
\begin{abstract}
  Any permutation statistic $f:\sym\to\CC$ may be represented uniquely
  as a, possibly infinite, linear combination of (classical)
  permutation patterns: $f= \Sigma_\tau\lambda_f(\tau)\tau$. To provide
  explicit expansions for certain statistics, we introduce a new type
  of permutation patterns that we call mesh patterns.  Intuitively,
  an occurrence of the mesh pattern $p=(\pi,R)$ is
  an occurrence of the permutation
  pattern $\pi$ with additional restrictions specified by $R$ on the
  relative position of the entries of the occurrence.  We show that,
  for any mesh pattern $p=(\pi,R)$, we have $\lambda_p(\tau) =
  (-1)^{|\tau|-|\pi|}\dual{p}(\tau)$ where $\dual{p}=(\pi,R^c)$ is the
  mesh pattern with the same underlying permutation as $p$ but with
  complementary restrictions. We use this result to expand some well
  known permutation statistics, such as the number of left-to-right
  maxima, descents, excedances, fixed points, strong fixed points, and
  the major index. We also show that alternating permutations, Andr\'e
  permutations of the
  first kind and simsun permutations occur naturally as permutations
  avoiding certain mesh patterns. Finally, we provide new natural Mahonian 
  statistics.
\end{abstract}

\begin{center}
\end{center}

\maketitle

\thispagestyle{empty}

\section{Introduction}

\subsection{Mesh patterns}
Let $[a,b]$ be the integer interval $\{ i\in\ZZ : a\leq i\leq
b\}$. Denote by $\sym_n$ the set of permutations of $[1,n]$. A
\emph{mesh pattern} is a pair
$$p=(\pi,R)\,\text{ with $\pi\in\sym_k$ and $R\subseteq [0,k]\times [0,k]$.}
$$ An example is $p=\big(3241, \{(0,2), (1,3), (1,4), (4,2),
(4,3)\}\big)$. To depict this mesh pattern we plot the points $(i,\pi(i))$
in a Cartesian coordinate system, and for each $(i,j)\in R$ we shade
the unit square with bottom left corner $(i,j)$:
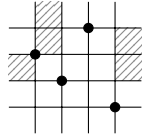
$$\pattern{}{ 4 }{ 1/3, 2/2, 3/4, 4/1 }{ 0/2, 1/3, 1/4, 4/2, 4/3 }
$$

Let $p=(\pi,R)$ be a mesh pattern with $k=|\pi|$, where $|\pi|$ denotes 
the number of letters in $\pi$, and let
$\tau\in\sym_n$.  We will think of $p$ as a function on permutations
that counts occurrences of $p$. Intuitively, $p(\tau)$ is the number
of ``classical" occurrences of $\pi$ in $\tau$ with additional
restrictions on the relative position of the entries of the occurrence
of $\pi$ in $\tau$. These restrictions say that no elements of $\tau$
are allowed in the shaded regions of the figure above. 
Formally, an
\emph{occurrence} of $p$ in $\tau$ is a subset $\omega$ of the plot of $\tau$,
$G(\tau)=\{(i,\tau(i)): i\in[1,n]\}$,  such that
there are order-preserving injections $\alpha,\beta:[1,k]\to [1,n]$
satisfying two conditions that we shall now describe. The first
condition is that $\omega$ is an occurrence of $\pi$ in the classical
sense. That is,
\begin{enumerate}
\item[(i)] $\omega = \big\{(\alpha(i),\beta(j)) : (i,j)\in G(\pi)\big\}$.
\end{enumerate}
Define 
$R_{ij}=[\alpha(i)+1, \alpha(i+1)-1]\times [\beta(j)+1,\beta(j+1)-1]$ 
for $i,j\in [0,k]$, where $\alpha(0)=\beta(0)=0$ and
$\alpha(k+1)=\beta(k+1)=n+1$.  Then the second condition is
\begin{enumerate}
\item[(ii)] if $(i,j)\in R$ then $R_{ij}\cap G(\tau)=\emptyset$. 
\end{enumerate}
Classical~\cite{SiSc}, vincular~\cite{BaSt} and bivincular~\cite{BCDK}
patterns can all be seen as special mesh patterns: $p=(\pi,R)$ is a
classical pattern if $R=\emptyset$; $p$ is a vincular pattern if $R$
is a union of vertical strips, $\{i\}\times [0,|\pi|]$; $p$ is a
bivincular pattern if $R$ is a union of vertical strips and horizontal
strips, $[0,|\pi|]\times\{i\}$. An example is provided by the following bivincular
pattern which has been studied by Bousquet-M\'elou \emph{et al.}~\cite{BCDK}:
$$\big(\,231,\,[0,3]\!\times\!\{1\}\cup \{1\}\!\times\![0,3]\,\big)
= \pattern{scale=\scl}{3}{1/2, 2/3, 3/1}{ 1/0, 1/1, 1/2, 1/3, 0/1, 2/1, 3/1}.
$$

It is also easy to write any barred pattern~\cite{West} with only
one barred letter as a mesh pattern. Indeed, if $\pi(i)$ is the only
barred letter of a given barred pattern $\pi$, then the corresponding
mesh pattern is $(\pi',\{(i-1, \pi(i)-1)\}$, where $\pi'$ is obtained
from $\pi$ by removing $\pi(i)$ and subtracting one from each letter
that is larger than $\pi(i)$. For instance, West~\cite{West}
characterized the permutations sortable by two passes through a stack
as those that avoid the classical pattern $2341$ and the barred pattern
$3\bar{5}241$. So, in terms of mesh patterns, it is the set of
permutations that avoid
$$
\pattern{scale=\scl}{4}{ 1/2,2/3,3/4,4/1 }{}\quad\text{and}\quad
\pattern{scale=\scl}{4}{ 1/3,2/2,3/4,4/1 }{ 1/4 }.
$$

The number of saturated chains in Young's
Lattice from $\hat 0$ (the empty partition) to a partition $\lambda$
is the number of standard Young tableaux of shape $\lambda$, and the
total number of saturated chains from $\hat 0$ to rank $n$ is the
number of involutions in $\sym_n$. Bergeron \emph{et al.}~\cite{BBD}
studied a composition analogue of Young's lattice. They gave an
embedding of the saturated chains from $\hat 0$ to rank $n$ into
$\sym_n$, and they characterized the image under this embedding as
follows: Let $T(\pi)$ be the increasing binary tree corresponding to
$\pi$.\footnote{ If $\pi$ is the empty word then $T(\pi)$ is the empty
  tree. Otherwise, write $\pi=\sigma a \tau$ with $a=\min(\pi)$, then
  $T(\pi)$ is the binary tree with root $a$ attached to a left subtree
  $T(\sigma)$ and a right subtree $T(\tau)$.  } Then $\pi\in\sym_n$
encodes a saturated chain from $\hat 0$ to rank $n$ if and only if
for any vertex $v$ of $T(\pi)$ that do not belong to the leftmost
branch of $T(\pi)$ and has two sons, the label of the left son is less
that the label of the right son. There is a unique smallest permutation
not satisfying this, namely $1423$; the corresponding increasing binary tree is
$$
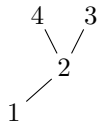
\begin{tikzpicture}[scale=0.35, inner sep=2pt]
  \node (1) at (0.1,0.3) {$1$};
  \node (2) at (2,2) {$2$};
  \node (3) at (3,4) {$3$};
  \node (4) at (1,4) {$4$};
  \draw (1) -- (2) -- (3);
  \draw (2) -- (4);
\end{tikzpicture}
$$ 
In terms of mesh patterns the permutations encoding saturated chains
from $\hat 0$ to rank $n$ are precisely those that avoid
$$\pattern{scale=\scl}{4}{1/1,2/4,3/2,4/3}
{1/0,1/1,1/2,1/3,2/0,2/1,2/2,2/3,3/0,3/1,3/2,3/3}.
$$

By $p(\tau)$ we shall denote the number of occurrences of $p$ in
$\tau$, thus regarding $p$ as a function from $\sym=\cup_{n\geq
  0}\sym_n$ to $\NN$. We will now explain how a few well known 
permutation statistics may be expressed in terms  of mesh patterns. A
\emph{left-to-right maximum} of $\tau$ is an index $j$ such that
$\tau(i)<\tau(j)$ for $i<j$. We write $\lmax(\tau)$ for the number of
left-to-right maxima in $\tau$. A \emph{descent} is an $i$ such that
$\tau(i)>\tau(i+1)$. The number of descents is denoted
$\des(\tau)$. An \emph{inversion} is a pair $i<j$ such that
$\tau(i)>\tau(j)$. The number of inversions is denoted $\inv(\tau)$.
For permutations $\alpha$ and $\beta$, let their \emph{direct sum} be
$\alpha\oplus\beta=\alpha\beta'$, where $\beta'$ is obtained from
$\beta$ by adding $|\alpha|$ to each of its letters, and juxtaposition
denotes concatenation. We say that $\tau$ has $k$ \emph{components},
and write $\comp(\tau)=k$, if $\tau$ is the direct sum of $k$, but not
$k+1$, non-empty permutations. We have
$$
\lmax = \pattern{scale=\scl}{1}{1/1}{ 0/1 };\quad
\inv  = \pattern{scale=\scl}{2}{1/2, 2/1}{};\quad
\des  = \pattern{scale=\scl}{2}{1/2, 2/1}{ 1/0, 1/1, 1/2 };\quad
\comp = 
\pattern{scale=\scl}{1}{1/1}{ 0/0, 0/1, 1/0, 1/1 } +
\pattern{scale=\scl}{2}{1/1, 2/2}{ 0/1, 0/2, 1/1, 1/2, 2/0 }.
$$

\subsection{Permutation statistics and an incidence algebra}
In what follows we will often simply write $\pi$ instead of
$(\pi,\emptyset)$, so $\inv=21$. We shall see that any function $\stat :
\sym \rightarrow \CC$ may be represented uniquely as a (possibly
infinite) sum $\stat = \sum_{\pi \in \sym}\lambda(\pi)\pi$, where
$\{\lambda(\pi)\}_{\pi \in \sym}\subset \CC$.

Let $Q$ be a locally finite poset, and let $\Int(Q)=\{ (x,y) \in
Q\times Q : x \leq y\}$. Recall that the \emph{incidence algebra},
$I(Q)$, of $(Q,\leq)$ over $\CC$ is the $\CC$-algebra of all functions
$F:\Int(Q) \to\CC$ with multiplication (convolution) defined by
$$(FG)(x,z) = \sum_{x\leq y\leq z}F(x,y)G(y,z),
$$ and identity, $\delta$, defined by $\delta(x,y) = 1$ if $x = y$,
and $\delta(x,y) = 0$ if $x \neq y$; see for example \cite[Sec. 3.6]{St}.

Define a partial order on $\sym$ by $\pi\leq\sigma$ in $\sym$
if $\pi(\sigma)>0$.  Define $P\in I(\sym)$ by
$$P(\pi,\sigma) = \pi(\sigma).
$$ Note that $P$ is invertible because $P(\pi,\pi)=1$, see
\cite[Prop. 3.6.2]{St}. Therefore, for any permutation statistic,
$\stat:\sym\to\CC$, there are unique scalars
$\{\lambda(\sigma)\}_{\sigma \in \sym} \subset \CC$ such that
\begin{equation}\label{stat}
  \stat = \sum_{\sigma\in\sym}\lambda(\sigma) \sigma.
\end{equation}
In other words, any permutation statistic can be written as a unique,
typically infinite, formal linear combination of (classical) patterns.
Indeed, $I(\sym)$ acts on the right of $\CC^{\sym}$ by
$$(f\ast F)(\pi) = \sum_{\sigma\leq\pi}f(\sigma)F(\sigma,\pi).
$$ Thus \eqref{stat} is equivalent to $\stat = \lambda \ast P$ and,
since $P$ is invertible, $\lambda = \stat \ast P^{-1}$.

\section{The Reciprocity Theorem}

The following mysterious looking identity for the descent statistic
$$\des
=\!\!\mathop{\sum_{\pi\in\sym}}_{\pi(1) > \pi(|\pi|)}\!\!(-1)^{|\pi|}\pi
$$ is an instance of what we call the Reciprocity Theorem for mesh
patterns. It tells us what the coefficients
$\{\lambda(\sigma)\}_{\sigma \in \sym} $ are in the special case when
$\stat=p$, a mesh pattern.
The Reciprocity Theorem may be viewed as a justification for the
introduction of mesh patterns. Indeed it shows that to describe the
coefficients of ``generalized permutation patterns'' requires that the
set of patterns is closed under taking complementary restrictions.

\begin{theorem}[Reciprocity]
  Let $p = (\pi,R)$ be a mesh pattern and let $\dual{p} = (\pi,R^c)$,
  where $R^c = [0,|\pi|]^2\setminus R$. Then
  $$p = \sum_{\sigma\in\sym}\lambda(\sigma)\sigma,\text{ where }
  \lambda(\sigma) = (-1)^{|\sigma|-|\pi|}\dual{p}(\sigma).
  $$
\end{theorem}

\begin{proof}
  We need to prove that $\dual{p}(\tau) = \sum_{\sigma\leq\tau}
  (-1)^{|\pi|-|\sigma|}p(\sigma)\sigma(\tau)$ for all $\tau \in
  \sym$. We will think of an occurrence of a pattern $p$ in $\sigma$
  as the corresponding subword of $\sigma$. The right-hand side may be
  written as
  \begin{equation}\label{star}
    \sum_{(\opi, \osi)}(-1)^{|\pi|-|\sigma|},
  \end{equation}
  in which sum is over all pairs $(\opi,\osi)$ where $\opi$ is a
  occurrence of $p$ in $\osi$ and $\osi$ is an occurrence of some 
  $\sigma \leq \tau$. Expression~\eqref{star} may, in turn, be written as
  $$\sum_{\opi}(-1)^{|\pi|}\mu(\opi),
  $$ where $\mu(\opi)$ is the contribution from a given occurrence
  $\opi$ of $\pi$. Given $\opi$, to create a pair $(\opi,\osi)$ we
  include any elements which are in squares \emph{not} indexed by the
  restrictions $R$. Let $X(\opi)$ be the set of such elements. Hence
  $$\mu(\opi) = \sum_{S\subseteq X(\opi)} (-1)^{|\pi|+|S|}.
  $$ Thus $\mu(\opi)=0$ unless $X(\opi)=\emptyset$. Clearly
  $X(\opi)=\emptyset$ if and only if $\opi=\osi$ and $\osi$ is an
  occurrence of $\dual{p}$. Consequently,
  $\sum_{\opi}(-1)^{|\pi|}\mu(\opi) = \dual{p}(\tau)$, as claimed.
\end{proof}

\begin{corollary}[Inverse Theorem]
  The inverse of $P$ in $I(\sym)$ is given by
  $$ P^{-1}(\pi,\tau) = (-1)^{|\tau|-|\pi|}P(\pi,\tau).
  $$
  Equivalently, if $f,g : \sym \rightarrow \CC$, then 
  $$
  f(\pi) = \sum_{\sigma \leq \pi} g(\sigma)\sigma(\pi), \quad\text{ for all } \pi\in\sym
  $$
  if and only if 
  $$ g(\pi) = \sum_{\sigma \leq \pi}
  f(\sigma)(-1)^{|\pi|-|\sigma|}\sigma(\pi), \quad\text{ for all }
  \pi \in \sym.
  $$
\end{corollary}
\begin{proof}
  For $\pi\in\sym_k$, let $p=(\pi,[0,k]\times[0,k])$. Then
  $\dual{p}=(\pi,\emptyset)$ and $p(\tau)=\delta(\pi,\tau)$, so by
  the Reciprocity Theorem,
  $$ \delta(\pi,\tau) = \sum_{\pi\leq\sigma\leq\tau}(-1)^{|\sigma|-|\pi|}
  P(\pi,\sigma)P(\sigma,\tau),
  $$
  from which the result follows.
\end{proof}

\section{Expansions of some permutation statistics}

Babson and Steingr\'{\i}msson's~\cite{BaSt} classification of Mahonian
statistics is in terms of vincular patterns. For example, the
\emph{major index}, $\maj$, can be defined as
$$
 (21,\{1\}\times[0,2]) +
(132,\{2\}\times[0,3]) +
(231,\{2\}\times[0,3]) +
(321,\{2\}\times[0,3]),
$$
or in pictures:
\begin{align*}
\maj &=
\pattern{scale=\scl}{ 2 }{ 1/2, 2/1 }{ 1/0, 1/1, 1/2 } +
\pattern{scale=\scl}{ 3 }{ 1/1, 2/3, 3/2 }{ 2/0, 2/1, 2/2, 2/3 } +
\pattern{scale=\scl}{ 3 }{ 1/2, 2/3, 3/1 }{ 2/0, 2/1, 2/2, 2/3 } +
\pattern{scale=\scl}{ 3 }{ 1/3, 2/2, 3/1 }{ 2/0, 2/1, 2/2, 2/3 }.
\intertext{By the Reciprocity Theorem we may represent the major index as $\maj=
\sum_{\pi \in \sym} \lambda(\pi)\pi$ where}
(-1)^{|\,\cdot\,|}\lambda(\,\cdot\,) &=
\pattern{scale=\scl}{ 2 }{ 1/2, 2/1 }
        { 0/0, 0/1, 0/2, 2/0, 2/1, 2/2 } -
\pattern{scale=\scl}{ 3 }{ 1/1, 2/3, 3/2 }
        { 0/0, 0/1, 0/2, 0/3, 1/0, 1/1, 1/2, 1/3, 3/0, 3/1, 3/2, 3/3 } -
\pattern{scale=\scl}{ 3 }{ 1/2, 2/3, 3/1 }
        { 0/0, 0/1, 0/2, 0/3, 1/0, 1/1, 1/2, 1/3, 3/0, 3/1, 3/2, 3/3 } -
\pattern{scale=\scl}{ 3 }{ 1/3, 2/2, 3/1 }
        { 0/0, 0/1, 0/2, 0/3, 1/0, 1/1, 1/2, 1/3, 3/0, 3/1, 3/2, 3/3 }.
\intertext{This last expression simplifies to}
\lambda(\pi) &=
\begin{cases}
  1         & \text{if $\pi=21$},\\
  (-1)^{n}   & \text{if $\pi(2)<\pi(n)<\pi(1)$},\\
  (-1)^{n+1} & \text{if $\pi(1)<\pi(n)<\pi(2)$},\\
  0         & \text{otherwise,}
\end{cases}
\end{align*}
where $n=|\pi|$. 

Now, let us plot the values of $\pi \in \sym$ in a Cartesian
coordinate system and locate the position of $x=\pi(j)$:
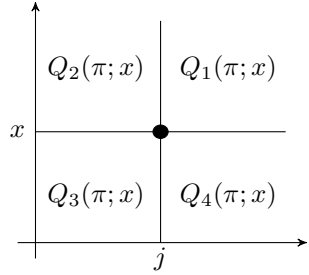
$$
\begin{tikzpicture}[>=stealth', yscale=0.42, xscale=0.47, baseline=(x.base)]
  \node at (-.5, 3.5) {$x$};
  \node at (3.5, -.5) {$j$};
  
  \draw[black] (3.5,0)--(3.5,7);
  \draw[black] (0,3.5)--(7,3.5);
  
  \filldraw[black] (3.5,3.5) circle (6pt);
  \node (x) at (3.5,3.5) {};
  \node at (1.7,5.5) {$Q_2(\pi; x)$};
  \node at (5.4,5.5) {$Q_1(\pi; x)$};
  \node at (5.4,1.5) {$Q_4(\pi; x)$};
  \node at (1.7,1.5) {$Q_3(\pi; x)$};
  \draw[->] (-0.5,0)--(7.6,0);
  \draw[->] (0,-0.5)--(0,7.6);
\end{tikzpicture}
$$
Here $Q_2(\pi;x)=\{\pi(i):i<j\text{ and }\pi(i)>x\}$, and the sets
$Q_k(\pi;x)$ for $k=1,3,4$ are defined similarly. Many permutation
patterns are defined in terms of the $Q_k$'s.
\begin{itemize}
\item $x$ is a \emph{left-to-right maximum} if
  $Q_2(\pi;x)=\emptyset$. Recall that $\lmax(\pi)$ denotes the number
  of left-to-right maxima in $\pi$;
\item $x$ is a \emph{fixed point} if
  $|Q_2(\pi;x)|=|Q_4(\pi;x)|$. Denote by $\fix(\pi)$ the number of
  fixed points in $\pi$;
\item The \emph{excess} of $x$ in $\pi$ is
  $x-j=|Q_4(\pi;x)|-|Q_2(\pi;x)|$. For $k \in \ZZ$, let $\exc_k(\pi)$
  be the number of $x$ in $\pi$ for which
  $|Q_4(\pi;x)|-|Q_2(\pi;x)|=k$;
\item $x$ is an \emph{excedance top} if $|Q_4(\pi;x)| >
  |Q_2(\pi;x)|$. Denote by $\exc(\pi)$ the number of excedance tops in
  $\pi$;
\item $x$ is a \emph{strong fixed point} if
  $Q_2(\pi;x)=Q_4(\pi;x)=\emptyset$, see \cite[Ex. 1.32b]{St}. Denote by
  $\sfix(\pi)$ the number of strong fixed points in $\pi$;
\item $x$ is a \emph{skew strong fixed point} if
  $Q_1(\pi;x)=Q_3(\pi;x)=\emptyset$. Denote by $\ssfix(\pi)$ the number of
  skew strong fixed points in $\pi$. Moreover, let $\SSF(\pi)$ be the
  set of skew strong fixed points in $\pi$.
\end{itemize}  

\begin{proposition}
  $$ \lmax 
  =\!\!\mathop{\sum_{\pi\in\sym}}_{\pi(|\pi|) =1}\!\!(-1)^{|\pi|-1}\pi.
  $$
\end{proposition}

\begin{proof}
  The result follows from the Reciprocity Theorem because
  the function $\pi\mapsto\chi(\pi(|\pi|)=1)$ equals
  $\pattern{scale=\scl}{1}{1/1}{ 0/0, 1/0, 1/1 }$ and $\lmax =
  \pattern{scale=\scl}{1}{1/1}{ 0/1 }$.
\end{proof}

\begin{proposition}
  Let $k \in \ZZ$. Then
  $$
  \exc_k= \sum_{\pi}\left((-1)^{|\pi|-k-1}\sum_{x\in\SSF(\pi)}\binom{|\pi|-1}{x-k-1}\right)\pi.
  $$
  In particular 
  $$\fix = \sum_{\pi}\left((-1)^{|\pi|-1}\sum_{x\in\SSF(\pi)}\binom{|\pi|-1}{x-1}\right)\pi
  $$
  and 
  $$
  \exc =  \sum_{\pi}\left((-1)^{|\pi|-2}\sum_{x\in\SSF(\pi)}\binom{|\pi|-2}{x-2}\right)\pi.
  $$
\end{proposition}
\begin{proof} 
  By the Inverse Theorem, $\exc_k=\sum_{\pi}\lambda_k(\pi)\pi$ where 
  $$
  \lambda_k(\pi) =
  \sum_{\sigma\leq\pi}(-1)^{|\pi|-|\sigma|}\exc_k(\sigma)\sigma(\pi).
  $$ 
  Let $\O_k(\pi)$ be the set of pair $(x,\omega)$ such that $\omega$
  is a subword of $\pi$ and $x$ is a letter of $\omega$ that has excess 
  $k$ in $\omega$. 
  Let $\al(\omega)$ denote the set of letters in $\omega$. Note that
  $(x,\omega)\in \O_k(\pi)$ if and only if $\big|Q_2(\pi;x) \cap
  \al(\omega)\big| +k= \big|Q_4(\pi;x) \cap \al(\omega)\big| $. Let
  $\alpha(x) = \min \left(Q_1(\pi;x) \cup Q_3(\pi;x)\right)$, 
  where $\min(\emptyset)=\infty$,  and
  define an involution $\Psi : \O_k(\pi) \rightarrow \O_k(\pi)$ by
  $$
  \Psi(x,\omega) = 
  \begin{cases}
    (x,\omega)                   &\text{if } \alpha(x)=\infty,\\
    (x,\omega\setminus\alpha(x)) &\text{if } \alpha(x)\in \al(\omega),\\
    (x,\omega\cup\alpha(x))      &\text{otherwise.}
  \end{cases}
  $$ 
  Here $\omega\setminus\alpha(x)$ denotes the word obtained by 
  deleting $\alpha(x)$, and $\omega\cup\alpha(x)$ the subword of $\pi$ obtained by 
  adding $\alpha(x)$ to $\omega$ at the correct position. 
  The mapping $\Psi$ is well-defined since the property of $x$
  having excess $k$ is invariant under adding elements to $Q_1(x)$ and
  $Q_3(x)$. Also, $\Psi$ reverses the sign, $(-1)^{|\pi|-|\omega|}$,
  on non fixed points. Moreover, $(x,\omega)$ is a fixed point if and
  only if $Q_1(x)= Q_3(x)=\emptyset$, that is, if and only if $x$ is a
  skew strong fixed point of $\pi$. It remains to determine the
  contribution of the skew strong fixed points $x$:
  \begin{align*}
    \mathop{\sum_{(x,\omega)\in \O_k(\pi)}}_{x\in\SSF(\pi)}(-1)^{|\pi|-|\omega|}
    &= \sum_j\binom{n-x}{j}\binom{x-1}{j+k}(-1)^{|\pi|-2j-k-1}\\[-2ex]
    &= (-1)^{|\pi|-k-1}\sum_j\binom{n-x}{j}\binom{x-1}{j+k}.
  \end{align*}
  Hence 
  $$(-1)^{|\pi|-k-1}\lambda_k(\pi) 
  = \sum_{x\in\SSF(\pi)}\sum_j\binom{n-x}{j}\binom{x-1}{j+k} 
  = \sum_{x \in\SSF(\pi)}\binom{|\pi|-1}{x-k-1},
  $$
  as claimed.
  The coefficient in front of $\pi$ in the expansion of $\exc$ is
  $\sum_{k\geq 1} \lambda_k(\pi)$. The expansion of $\exc$ then
  follows from
  $$
  \sum_{j=0}^k (-1)^j\binom n j = (-1)^k \binom {n-1} k.  
  $$
\end{proof}

\begin{proposition} 
  $$\sfix = \sum_\pi(-1)^{|\pi|-1}\ssfix(\pi) \pi.
  $$
\end{proposition}

\begin{proof}
  Because $\sfix=\pattern{scale=\scl}{1}{1/1}{0/1,1/0}$ and
  $\ssfix=\pattern{scale=\scl}{1}{1/1}{0/0,1/1}$ the result follows
  from the Reciprocity Theorem.
\end{proof}

\section{Euler numbers}

A permutation $\pi \in \sym_n$ is said to be \emph{alternating} if
$$\pi(1)> \pi(2)<\pi(3)>\pi(4)<\cdots.
$$ 
Clearly the set of alternating
permutations are exactly the permutations that avoid the vincular/mesh patterns
$$
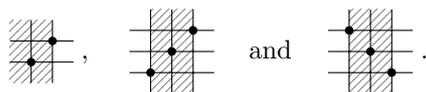
\pattern{scale=\scl}{2}{1/1,2/2}{0/0,0/1,0/2,1/0,1/1,1/2},\quad
\pattern{scale=\scl}{3}{1/1,2/2,3/3}{1/0,1/1,1/2,1/3,2/0,2/1,2/2,2/3}
\quad\text{and}\quad
\pattern{scale=\scl}{3}{1/3,2/2,3/1}{1/0,1/1,1/2,1/3,2/0,2/1,2/2,2/3}.
$$
In 1879, Andr\'e~\cite{Andre} showed that the
number of alternating permutations in $\sym_n$ is the
\emph{Euler number} $E_n$ given by
$$
\sum_{n \geq 0} E_n x^n/n! = \sec x + \tan x.
$$
There are several other sets of permutations enumerated by the
Euler numbers, see \cite{Sts}.  A \emph{simsun} permutation may be
defined as a permutation $\pi \in \sym_n$ for which for all $1\leq i
\leq n$, after removing the $i$ largest letters of $\pi$, the
remaining word has no double descents. In terms of mesh patterns, a
permutation is simsun if and only if it avoids the pattern
$$\SimSun = \simsun.
$$ 
simsun permutations are central in describing the action of the
symmetric group on the maximal chains of the partition lattice, and
the number of simsun permutations in $\sym_n$ is the Euler number
$E_{n+1}$, see \cite{Sund}. 

Another important class of permutations counted by the Euler numbers
are the \emph{Andr\'e permutations} of various kinds introduced by
Foata and Sch\"utzenberger~\cite{FSc73} and further studied by Foata and 
Strehl~\cite{FSt74}.
If $\pi \in \sym_n$ and $x=\pi(i) \in [1,n]$ let $\lambda(x),\rho(x)
\subset [1,n]$ be defined as follows. Let $\pi(0)=\pi(n+1)=-\infty$.
\begin{itemize}
\item $\lambda(x)= \{ \pi(k): j_0 < k <i\}$ where $j_0=\max\{ j : j<i
  \text{ and } \pi(j)<\pi(i)\}$, and
\item $\rho(x)= \{ \pi(k): i < k <j_1\}$ where $j_1=\min\{ j : i<j
  \text{ and } \pi(j)<\pi(i)\}$.
\end{itemize}
A permutation $\pi \in \sym_n$ is an \emph{Andr\'e permutation of the
  first kind} if
$$
\max \lambda(x) \leq \max \rho(x)
$$ 
for all $x \in [1,n]$, where $\max \emptyset =-\infty$. In
particular, $\pi$ has no double descents and $\pi(n-1) <
\pi(n)=n$. The concept of Andr\'e permutations of the first kind
extends naturally to permutation of any finite totally ordered
set. The following recursive description of Andr\'e permutations of
the first kind follows immediately from the definition.

\begin{lemma}\label{decom}
  Let $\pi \in \sym_{n}$ be such that $\pi(n)=n$. Write $\pi$ as
  the concatenation $\pi=L1R$. Then $\pi$ is an Andr\'e permutation of
  the first kind if and only if $L$ and $R$ are Andr\'e permutations
  of the first kind.
\end{lemma}

\begin{theorem}
  Let $\pi \in \sym_{n}$. Then $\pi$ is
  an Andr\'e permutation of the first kind if and only if it avoids
  $$
  \Andre = \andre
  \quad\text{and}\quad
  \pattern{scale=\scl}{2}{1/2, 2/1}{ 2/0, 2/1, 2/2 }.
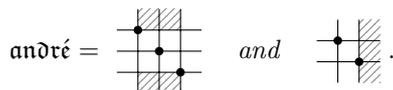
  $$
\end{theorem}
\begin{proof}
  Note that $\pi\in\sym_{n}$ avoids the second pattern if and only if $\pi(n)=n$.
  Write $\pi$ as $\pi = L1R$. Then $\pi$ avoids $\Andre$ if and only if
  $L$ and $R$ avoid the two patterns. 
  Thus the set of all permutations that avoid the two patterns
  have the same recursive description as the set of all
  Andr\'e permutation of the first kind, and hence the sets agree.
\end{proof}

\begin{corollary}
  $|\sym_n(\Andre)| =E_{n+1}$.
\end{corollary}

Note that Lemma~\ref{decom} immediately implies a version of the
recursion formula for the Euler numbers
$$
E_{n+1}= \sum_{k=0}^{n-1} \binom {n-1} k E_{k+1}E_{n-k-1},
$$
where $E_0=1$. 

Using a computer, it is not hard to see that up to trivial symmetries
the only essentially different mesh patterns $p=(321,R)$ such that
$|\sym_n(p)| =E_{n+1}$ for all $n$ are $\SimSun$ and $\Andre$.

\section{New Mahonian Statistics}

There are many ways of expressing the permutation statistic $\inv$ as
a sum of mesh patterns. For instance, 
\begin{equation}\label{decinv}
\inv = 
\pattern{scale=\scl}{2}{1/2,2/1}{2/2} +
\pattern{scale=\scl}{3}{1/2,2/1,3/3}{3/2,3/3}.
\end{equation}
Indeed, given $\pi\in\sym_n$ we may partition the 
set of inversions\footnote{Here the set of inversions means the set of 
  occurrences of the pattern $21$, not the positions of the inversions.}  
of $\pi$ into
two sets as follows. Let $I^+(\pi)$ denote the set of inversions
that play the role of $21$ in some occurrence of $213$, and let $I^-(\pi)$
denote the set of inversions that do not play the role of $21$ in any
occurrence of $213$. Then the first pattern in the right-hand-side of 
\eqref{decinv} agrees with $\pi \mapsto |I^-(\pi)|$, and the second 
with $\pi \mapsto |I^+(\pi)|$.

There is a similar decomposition of non-inversions:
\begin{equation}\label{decainv}
12 = 
\pattern{scale=\scl}{2}{1/1,2/2}{2/2} +
\pattern{scale=\scl}{3}{1/1,2/2,3/3}{3/2,3/3}.
\end{equation}
Now $A^+(\pi)$ is the set of non-inversions that play the role
of $12$ in some occurrence of $123$, and  $A^-(\pi)$ is the set of
non-inversions that do not play the role of $12$ in any occurrence of
$123$.

Can we mix the patterns in \eqref{decinv} and \eqref{decainv} and
still get a Mahonian statistic? Let
$$
\mix = 
\pattern{scale=\scl}{2}{1/1,2/2}{2/2} +
\pattern{scale=\scl}{3}{1/2,2/1,3/3}{3/2,3/3}.
$$
We will prove that $\mix$ is Mahonian.  Since $\mix(\pi)=|A^-(\pi)| +
|I^+(\pi)|$ and $12(\pi)=|A^-(\pi)|+|A^+(\pi)|$ it suffices to find a
bijection $\psi$ that fixes $|A^-(\pi)|$ and is such that
$|A^+(\psi(\pi))|= |I^+(\pi)|$. In fact we will prove more. Let $M,I
\subseteq [n]$ be such that $|M|=|I|$ and $n \in M\cap I$, and let
$\sym_n(M,I)$ be the set of permutations in $\sym_n$ that have
right-to-left maxima exactly at the positions indexed by $I$, and set
of values of the right-to-left maxima equal to $M$.  Let $\pi \in
\sym_n(M,I)$, and define two functions 
$\ci(\pi), \ca(\pi):[n]\setminus M \rightarrow [n-1]$ by
\begin{align*}
  \ci(\pi)(y) &= |I^+_y|,
  &\text{where }\, I^+_y &= \{ x : (y,x) \in I^+(\pi) \},\text{ and}\\
  \ca(\pi)(y) &= |A^+_y|,
  &\text{where }\, A^+_y &= \{ x : (x,y) \in A^+(\pi) \}.
\end{align*}

\begin{lemma}\label{code}
  Let $M,I \subseteq [n]$ be such that $|M|=|I|$ and $n \in M\cap I$. Then 
  $$
  \ci(\pi) = \ci(\sigma) \text{ if and only if } \pi = \sigma,
  $$
  and
  $$
  \ca(\pi) = \ca(\sigma) \text{ if and only if } \pi = \sigma,
  $$
  for all $\pi, \sigma \in \sym_n(M,I)$.
\end{lemma}
\begin{proof}
  Suppose that we know $\ci(\pi)$, and let $[n]\setminus M =
  \{s_1<\cdots<s_k\}$. Then $\ci(\pi)(s_k)$ tells us the position of
  $s_k$ in $\pi$, and recursively we can read off the position of
  $s_i$ from $\ci(\pi)(s_i)$, given that we know the positions of
  $s_{i+1}, \ldots, s_k$. Hence we can reconstruct $\pi$ from
  $\ci(\pi)$.

  Suppose that $s_1, \ldots, s_j$ are the elements of $[n]\setminus M$
  that are smaller than $\min(M)$. Then $\ca(\pi)(s_j)$ tells us the
  position of $s_j$ in $\pi$, and recursively we can read off the
  position of $s_i$ from $\ca(\pi)(s_i)$, given that we know the
  positions of $s_{i+1}, \ldots, s_j$. We can continue in the same way
  to read off the positions of the elements of $[n]\setminus M$ that
  are between $\min(M)$ and $\min(M\setminus \{\min(M)\})$ in
  size. Continuing this procedure we will recover $\pi$ from
  $\ca(\pi)$.
\end{proof}

\begin{theorem}\label{MI}
  Let $M,I \subseteq [n]$ be such that $|M|=|I|$ and $n \in M\cap
  I$. There is an involution $\psi : \sym_n(M,I) \rightarrow
  \sym_n(M,I)$ such that
  $$
  \big(\ca(\pi),\,\ci(\pi)\big) = \big(\ci(\psi(\pi)),\,\ca(\psi(\pi))\big),
  $$ for all $\pi \in \sym_n(M,I)$. Moreover, $\psi$ fixes $|A^-(\pi)|$.
\end{theorem}
\begin{proof}
  Let $M=\{m_1< \cdots< m_k\}$ and let $B_i$ be the set of entries of
  $\pi$ that are smaller than and to the left of $m_i$.  For $S
  \subseteq [n]$, let $\psi_S(\pi)$ be the permutation obtained by
  reversing the subword of $\pi$ that is a permutation on $S$. Define
  $\psi$ by
  $$
  \psi = \psi_{B_{1}} \circ \psi_{B_2\cap B_{1}} 
  \circ  \cdots \circ \psi_{B_{k-1}} \circ \psi_{B_k\cap B_{k-1}} \circ \psi_{B_k}. 
  $$
  For instance, with $\pi=125634$ we have $B_1=\{1,2,3\}$, $B_2=\{1,2,5\}$ and
  \begin{align*}
    \psi(\pi) 
    &= \psi_{B_1}\circ\psi_{B_1\cap B_2}\circ\psi_{B_2}(125634) \\
    &= \psi_{B_1}\circ\psi_{B_1\cap B_2}(521634) \\
    &= \psi_{B_1}(512634) \\
    &= 532614.
  \end{align*}
  It is easy to see that $\psi : \sym_n(M,I) \rightarrow \sym_n(M,I)$ and
  that $\psi$ fixes $|A^-(\pi)|$.

  For fixed $y$ we want to show that $(|A^+_y|, |I^+_y|) \mapsto
  (|I^+_y|, |A^+_y|)$ under $\psi$. If $S$ is a subset of a $B_j$ that
  does not contain $y$ then $|A^+_y|$ and $|I^+_y|$ are unchanged
  under $\psi_S$.  Let $r$ be the largest index for which $y \in B_r$,
  and let $s$ be the smallest index for which $y \in B_s$. Write
  $\psi$ as $\psi = \alpha \circ \beta \circ \gamma$ where
  $$
  \beta = \psi_{B_{s}} \circ \psi_{B_{s+1}\cap B_{s}} 
  \circ \cdots \circ \psi_{B_{r-1}} \circ \psi_{B_r\cap B_{r-1}} \circ \psi_{B_r}. 
  $$
  Now $y$ is not moved, and $(|A^+_y|, |I^+_y|)$ is conserved, by
  $\gamma$. Moreover $y$ will remain outside $B_1 \cup \cdots \cup
  B_{s-1}$ when we apply $\alpha$ after the action of $\beta \circ
  \gamma$.  Hence it remains to consider the effect on $(|A^+_y|,
  |I^+_y|)$ under $\beta$. Suppose $x \in A^+_y$. Then $\beta$
  switches $(x,y)$ an odd number of times. Hence $x \in I^+_y$ after
  having applied $\beta$. Since $A^+_y \cup I^+_y \subseteq B_s$, the
  set $A^+_y \cup I^+_y$ remains unchanged under the mappings defining
  $\beta$. It follows that $(A^+_y, I^+_y) \mapsto (I^+_y, A^+_y)$
  under $\beta$. We have proved that $ (\ca(\pi), \ci(\pi)) =
  (\ci(\psi(\pi)), \ca(\psi(\pi))) $. By Lemma~\ref{code} $\psi$ is an
  involution.
\end{proof}

Using a computer we have searched for Mahonian statistics that are of
the form $(12, R) + (\pi,S)$, with $\pi\in\sym_3$, $R\subseteq
[0,2]\times[0,2]$ and $S\subseteq [0,3]\times [0,3]$. Up to trivial
symmetries one new Mahonian statistic apart from $\mix$ was found. Namely
$$\smallskip
\mix'=
\pattern{scale=\scl}{2}{1/1,2/2}{1/2} +
\pattern{scale=\scl}{3}{1/2,2/3,3/1}{1/3,2/3}.\smallskip
$$
Again, $\mix'$ is, in a sense, a mix of $\inv$ and $12$. To be more precise, let
$$
S_1=\pattern{scale=\scl}{2}{1/1,2/2}{1/2},\quad
S_2=\pattern{scale=\scl}{2}{1/2,2/1}{1/2},\quad
T_1=\pattern{scale=\scl}{3}{1/1,2/3,3/2}{1/3,2/3}\quad\text{and}\quad
T_2=\pattern{scale=\scl}{3}{1/2,2/3,3/1}{1/3,2/3}.
$$ Then $12=S_1+T_1$, $\inv=S_2+T_2$ and $\mix'=S_1+T_2$. 
We note here that $S_1$ and $T_1$ have appeared before in the
literature: $S_1$ measure the major cost of the in-situ permutation
algorithm~\cite{KPT, Knuth}; the statistic $\inv + T_1$ is identical
to $\mathrm{lbsum}$ in \cite{DR}.

Let
$\tpqbin{n}{k}= [n]!/([k]![n-k]!)$ denote the usual $q_1,q_2$-binomial
coefficient, where $[n]!=[1][2]\dots[n]$ and
$[n]=q_1^{n-1} + q_1^{n-2}q_2 +\dots+ q_1q_2^{n-2} + q_2^{n-1}$.
The $q_1,q_2$-derivative of a function $f(x)$ is defined by
$$
\left(\frac{d}{dx}\right)_{q_1,q_2}f(x) = \frac{f(q_1x)-f(q_2x)}{q_1x - q_2x}.
$$
Let
$$
F_n  = \sum_{\pi\in\sym_n}p_1^{S_1(\pi)}p_2^{S_2(\pi)}q_1^{T_1(\pi)}q_2^{T_2(\pi)}
$$ record the joint distribution of the four
permutation statistics $S_1$, $S_2$, $T_1$, and $T_2$, and let
$\F(x) = \sum_{n\geq 0} F_n x^n / [n]!$ be the
$q_1,q_2$-exponential generating function for $\{F_n\}_{n\geq
  0}$. 

\begin{theorem}\label{F_n}
  For $n\geq 0$,
  $$F_{n+1} = \sum_{k=0}^n \pqbin{n}{k}p_1^kp_2^{n-k}F_kF_{n-k}.
  $$
Moreover $\displaystyle{\left(\frac{d}{dx}\right)_{q_1,q_2}\!\F(x) = \F(p_1 x)\F(p_2 x)}.$ 
\end{theorem}
\begin{proof}
  Let $\pi\in\sym_{n+1}$ and write $\pi=L(n+1)R$. Let $A=\al(L)\subseteq[n]$
  be the letters in $L$. Then
  \begin{align*}
    S_1(\pi) &= S_1(L) + S_1(R) + k, \\
    S_2(\pi) &= S_2(L) + S_2(R) + n-k, \\
    T_1(\pi) &= T_1(L) + T_1(R) + k(n-k) - \gamma(A), \\
    T_2(\pi) &= T_2(L) + T_2(R) + \gamma(A),
  \end{align*}
  where $k=|A|$ and
  $\gamma(A)=\{(i,j): i\in A, j\in[n]\setminus A, i>j \}$. Thus
  $$F_{n+1} =
  \sum_{k=0}^n\left(\sum_{A\in\binom{[n]}{k}}
  q_1^{k(n-k)-\gamma(A)}q_2^{\gamma(A)}\right) p_1^k p_2^{n-k} F_k F_{n-k}.
  $$
  We also have, e.g. from \cite[Prop 1.3.17, $m=2$, $a_1=k$, $a_2=n-k$]{St},
  that
  $$\sum_{A\in\binom{[n]}{k}} q_1^{k(n-k) - \gamma(A)}q_2^{\gamma(A)} = \pqbin{n}{k},
  $$
  which proves the recursion formula. The equation for the $q_1, q_2$-exponential generating function follows. 
\end{proof}
The fact that $\mix'$ is Mahonian now follows from the symmetry of the variables in the recursion formula. 
\begin{corollary}
  The statistic $\mix'$ is Mahonian.
\end{corollary}

\end{document}